\documentclass[preprint,11pt]{elsarticle}
\usepackage{breakurl}             %
\usepackage{amssymb,amsmath,amsfonts,amsthm,mathrsfs,srcltx}
\usepackage{multirow}

\newtheorem{theorem}{Theorem}
\newtheorem{lemma}[theorem]{Lemma}
\newtheorem{proposition}[theorem]{Proposition}
\newtheorem{corollary}[theorem]{Corollary}

\newcommand{\nats}{{\mathbb N}}

\theoremstyle{definition}
\newtheorem{example}[theorem]{Example}

\journal{}

\date{}

\begin{document}
\begin{frontmatter}

\title{Abelian returns in Sturmian words}



\author[label3,label4]{Svetlana Puzynina\corref{cor1}\fnref{label1}}
  \ead{svepuz@utu.fi}

 \author[label3,label5]{Luca Q. Zamboni\fnref{label2}}
  \ead{zamboni@math.univ-lyon1.fr}

  \fntext[label1]{Partially
supported by grant no. 251371 from the Academy of Finland, by
Russian Foundation of Basic Research (grant 10-01-00424) and by RF President grant for young scientists (MK-4075.2012.1).}
  \fntext[label2]{Partially supported by a grant from the Academy of Finland and by
ANR grant {\sl SUBTILE}.}
\address[label3]{University of Turku, Finland}
\address[label4]{Sobolev Institute of Mathematics, Russia}
\address[label5]{Universit\'e de Lyon 1, France}

\begin{abstract}
Return words constitute a powerful tool for studying symbolic
dynamical systems.  They may be regarded as a discrete analogue of
the first return map in dynamical systems. In this paper we
investigate  two abelian variants of the notion of return word,
each of them gives rise to a new characterization of Sturmian
words. We prove that a recurrent infinite word is Sturmian if and
only if each of its factors has two or three abelian (or
semi-abelian) returns. We study the structure of abelian returns
in Sturmian words and give a characterization of those factors
having exactly two abelian returns. Finally we discuss connections
between abelian returns and periodicity in words.
\end{abstract}


\end{frontmatter}

\section{Introduction}

Let $w \in A^{\nats} $ be an infinite word with values in a finite
alphabet $A.$  The {\it (factor) complexity function} $p:\nats
\rightarrow \nats$ assigns to each $n$ the number of distinct
factors of $w $ of length $n.$ A fundamental result of Hedlund and
Morse \cite{MoHe1} states that a word $w $ is ultimately periodic
if and only if for some $n$ the complexity $p(n)\leq n.$ Infinite words
of complexity $p(n)=n+1$ are called {\it Sturmian words.} The most
studied Sturmian word is the so-called Fibonacci word
\[01001010010010100101001001010010\ldots\]
fixed by the morphism $0\mapsto 01$ and $1\mapsto 0.$ In \cite
{MoHe2} Hedlund and Morse showed that each Sturmian word may be
realized geometrically by an irrational rotation on the circle.
More precisely, every Sturmian word is obtained by coding the
symbolic orbit of a point $x$ on the circle (of circumference one)
under a rotation by an irrational angle $\alpha $ where the circle
is partitioned into two complementary intervals, one of length
$\alpha $ and the other of length $1-\alpha .$ And conversely each
such coding gives rise to a Sturmian word. The irrational $\alpha
$ is called the {\it slope} of the Sturmian word. An alternative
characterization using continued fractions was given by Rauzy in
\cite{Ra1} and \cite{Ra2}, and later by Arnoux and Rauzy in \cite
{ArRa}. Sturmian words admit various other types of
characterizations of geometric and combinatorial nature (see for
instance \cite{lothaire}). For example they are characterized by
the following \emph{balance} property: A word $w $ is Sturmian if
and only if $w $ is a binary aperiodic (non-ultimately periodic)
word and $\left||u|_i-|v|_i\right|\leq 1$ for all factors $u$ and
$v$ of $w $ of equal length, and for each letter $i.$ Here $|u|_i$
denotes the number of occurrences of $i$ in $u.$\\

In this paper we develop and study two abelian analogues of the notion
of return word and apply it to characterize Sturmian words. Return
words constitute a powerful tool for studying various problems in
combinatorics on words, symbolic dynamical systems and number
theory. Given a factor $v$ of an infinite word $w,$ by a {\it return word}
to $v$ (in $w)$ we mean a factor $u$ of $w$ such that $uv$ is a factor of $w$
beginning and ending in $v$ and having no other (internal)
occurrence of $v.$ In other words the set of all return words to $v$
is the set of all distinct words beginning with an occurrence of $v$
and ending just before the next occurrence of $v$. The notion of
return words can be regarded as a discrete analogue of  the first
return map in dynamical systems. Many developments of the notion of
return words have been given: For example, return words are  used to
characterize primitive substitutive sequences \cite {durand, hz}.
Return words are used in  studying the transcendence of Sturmian or
morphic continued fractions \cite {adqz}. Return words were
fruitfully studied in the context of interval exchange
transformations (see \cite{vuillon1}). Words having a constant
number of return
words were considered in \cite{BPS}. 
 In \cite {fv} a generalization of the notion of balanced property
 for Sturmian words was introduced and the proof is based on return words.
 Return words are also used to characterize periodicity and Sturmian words.
 The following characterization was obtained by L. Vuillon in \cite {vuillon}:

\begin{theorem}  \label{returns}  {\rm \cite {vuillon}} A binary recurrent infinite word $w$ is Sturmian if and only if each factor $u$ of $w$  has
two returns in $w.$ \end{theorem}

\noindent In \cite {jv} the proofs were simplified and
return words were studied in the context of episturmian words.

\smallskip

Two words are said to be {\it abelian equivalent} if they are permutations of each
other. It is readily verified that this defines an equivalence relation on the set of all factors of an infinite word. Various abelian properties of words have been extensively investigated including abelian powers and their avoidance, abelian complexity and abelian
periods \cite {af, akp, dekking, keranen, rsz}. Given a factor $u$ of an infinite word $w,$  let
$n_1<n_2<n_3<\dots$ be all integers $n_i$ such that $w_{n_i}\dots
w_{n_{i}+|u|-1}$ is abelian equivalent $u.$ Then we call each  $w_{n_i}\dots
w_{n_{i}+|u|-1}$ a \emph{semi-abelian return} to $u.$
By an \emph{abelian return} to $u$  we mean an  abelian class of  $w_{n_i}\dots w_{n_{i+1}-1}.$ We note that in both cases these definitions depend only on the abelian class of $u.$
Each of these notions of abelian returns gives rise to a new
characterization of Sturmian words:


\begin{theorem} \label {main} A binary recurrent infinite word $w$ is Sturmian if and only
if each factor $u$ of $w$  has two or three abelian returns in $w.$
\end{theorem}


Surprisingly, Sturmian words admit
exactly the same characterization in terms of semi-abelian returns:

\begin{theorem}\label{semi-abelian} A binary recurrent infinite word $w$  is Sturmian if and only
if each factor $u$ of $w$  has two or three semi-abelian returns in $w$.
\end{theorem}

Although the above characterizations of Sturmian words are similar to the one given in Theorem
\ref {returns}, our methods differ considerably from those used in
\cite {jv, vuillon}.

\smallskip

The paper is organized as follows:  Section 2 is devoted to providing the necessary background and terminology relevant to the subsequent sections. In Section 3 we investigate
connections between abelian returns and periodicity. 
In Section 4 we study the structure of abelian returns in Sturmian
words. We prove that every factor of a Sturmian word has two or
three abelian returns (Proposition \ref{necessity}) and moreover, a
factor has two abelian returns if and only if it is singular
(Theorem \ref{singular}). In Section 5 we prove the sufficiency of
the condition on the number of abelian returns for a word to be
Sturmian (Corollary \ref {corollary_suf}). In Section 6 we prove Theorem \ref {semi-abelian}. 

\section{Preliminaries}

\subsection{Sturmian words and return words}

We begin by presenting some background on Sturmian words  and return words
and terminology which will be used later in the paper.

Given a finite non-empty set $\Sigma$ (called the alphabet), we
denote by $\Sigma^*$ and $\Sigma^{\omega}$,
respectively, the set of finite words and the set of (right) infinite
words over the alphabet $\Sigma$.
 A word $v$ is a \emph{factor} (resp. a \emph{prefix}, resp. a \emph{suffix}) of a word $w$, if there exist words $x$, $y$ such that $w=xvy$ (resp. $w=vy$, resp. $w=xv$).  The set of factors of a finite or infinite word $w$ is denoted by $F(w)$. Given a finite word $u = u_1 u_2 \dots u_n$ with $n \geq 1$ and
$u_i \in \Sigma$, we denote the length $n$ of $u$ by $|u|$. The
empty word will be denoted by $\varepsilon$ and we set
$|\varepsilon| = 0$. For each $a \in \Sigma$, we let $|u|_a$
denote the number of occurrences of the letter $a$ in $u$.
An infinite word $w$ is said to be \emph{$k$-balanced} if and only if
$||u|_a-|v|_a|\leq k$ for all factors $u,v$ of $w$ of equal length and all letters $a\in \Sigma.$
If $w$ is $1$-balanced, then we say that $w$ is {\it balanced}.

Two words
$u$ and $v$ in $\Sigma^*$ are said to be \emph{abelian equivalent}, denoted $u \sim_{ab} v$,
if and only if $|u|_a = |v|_a$ for all $a
\in \Sigma$. It is easy to see that abelian equivalence is indeed
an equivalence relation on $\Sigma^*$.

We say that a (finite or infinite) word $w$ is \emph{periodic}, if
there exists $T$ such that $w_{n+T}=w_n$ for every $n$. A right
infinite word $w$ is \emph{ultimately periodic} if there exist
$T$, $n_0$ such that $w_{n+T}=w_n$ for every $n\geq n_0$. A word
$w$ is \emph{aperiodic}, if it is not (ultimately) periodic. A
factor $u$ of $w$ is called \emph{right special} if both $ua$ and
$ub$ are factors of $w$ for some pair of distinct letters $a, b
\in \Sigma$. Similarly $u$ is called \emph{left special} if both
$au$ and $bu$ are factors of $w$ for some pair of distinct letters
$a, b \in \Sigma$. The factor $u$ is called \emph{bispecial} if it
is both right special and left special.

Sturmian words can be defined in many different ways. For example,
they are infinite words having the smallest factor complexity
among aperiodic words. 
By a celebrated result due to Hedlund and Morse \cite {MoHe1}, a
word is ultimately periodic if and only if its factor complexity
$p(n)$ is uniformly bounded. In particular, $p(n) < n$ for all $n$
sufficiently large. Sturmian words are exactly words whose factor
complexity $p(n) = n + 1$ for all $n \geq 0$. Thus, Sturmian words
are those aperiodic words having the lowest complexity. Since
$p(1) = 2$, it follows that Sturmian words are binary words. In
what follows, we denote the letters of a Sturmian word by $0$ and
$1$.

The condition $p(n) = n+1$ implies the existence of exactly one
right special and one left special factor of each length. The set
of factors of a Sturmian word is closed under reversal, so for
every length the right special factor is a reversed left special
factor, and bispecial factors are palindromes. Bispecial factors play a crucial role in Sturmian words.  \emph{Standard factors} of a Stumian word $w$ are letters and factors of the form $Bab$, where $a\neq b \in \{0, 1\}$ and $B$ is a bispecial factor of $w.$  A factor of a Sturmian word is called
\emph{singular} if it is the only factor in its abelian class.
It is well known that singular factors have the form $aBa$, where $a$ is a
letter and $B$ a bispecial factor. We will also use the notion of \emph{Christoffel word}. One of the ways to define Christoffel words is the following: they are factors of a Sturmian word of the form $aBb$ and letters.


In \cite {MoHe2} it is shown that each Sturmian
word may be realized measure-theoretically by an irrational rotation on the circle. That is, every
Sturmian word is obtained by coding the symbolic orbit of a point $x$ on the circle (of circumference
one) under a rotation by an irrational angle $\alpha$, $0 < \alpha < 1$, where the circle is partitioned into
two complementary intervals, one of length $\alpha$ and the other of length $1-\alpha$. And conversely
each such coding gives rise to a Sturmian word. The quantity $\alpha$ gives the frequency of letter $1$ in the Sturmian word defined by such rotation. Other widely used
characterizations are via mechanical words, cutting
sequences, Sturmian morphisms etc., see \cite {lothaire} for
further detail.

\medskip

Let $w=w_1w_2\dots$ be an infinite word. The word $w$ is
\emph{recurrent} if each of its factors occurs infinitely many times
in $w.$ In this case, for $u\in F(w)$, let $n_1<n_2<\dots$ be all
integers $n_i$ such that $u=w_{n_i}\dots w_{n_{i}+|u|-1}$. Then the
word $w_{n_i}\dots w_{n_{i+1}-1}$ is a \emph{return word} (or
briefly \emph{return}) of $u$ in $w$. An infinite word \emph{has $k$
returns}, if each of its factors has $k$ returns. The following
characterization of Sturmian words via return words was established
in \cite {vuillon}: A word is Sturmian if and only if each of its
factors has two returns (Theorem 1 in the Introduction).


Also there exists a simple characterization of periodicity via return words:

\begin{proposition} \label{per} {\rm \cite {vuillon}} A recurrent infinite word is ultimately periodic if and only if there exists a factor having exactly one return word. \end{proposition}

\subsection{Abelian and semi-abelian returns}

In this subsection we define the basic notions for the abelian
case. In particular, we introduce two abelian versions of the
notion of return word, abelian return and semi-abelian return.

For an infinite recurrent word $w$ and for $u\in F(w)$, let
$n_1<n_2<n_3<\dots$ be all integers $n_i$ such that $w_{n_i}\dots
w_{n_{i}+|u|-1}\sim_{ab} u$. Then each $w_{n_i}\dots
w_{n_{i}+|u|-1}$ is called a {\it semi-abelian return} to the
abelian class of $u.$ By an \emph{abelian return} 
to the abelian class of $u$ we mean an abelian class of
$w_{n_i}\dots w_{n_{i+1}-1}.$  So the number of abelian returns is
the number of distinct abelian classes of semi-abelian returns.
Hence for every factor $u$ in an infinite word $w$ the number of
abelian returns to the abelian class of $u$ is less or equal to
the number of semi-abelian returns to the abelian class of $u.$
For brevity in the further text we often say (semi-)abelian return
to factor $u$ meaning the abelian class of $u$. We will often
denote abelian returns by an element from the abelian equivalence
class, that is by a semi-abelian return from the class.

\begin{example}\label{example1} Consider the Thue-Morse word $$t=0110100110010110\dots$$
 fixed by the morphism $\mu$: $\mu (0)=01$, $\mu
(1)=10$. The abelian class of $01$ consists of two words $01$ and
$10$. Consider an occurrence of $01$ starting at position  $i$,
i.e., $t_i=0$, $t_{i+1}=1$. It can be followed by either $0$ or
$10$, i.e. we have either $t_{i+2}=0$ or $t_{i+2}=1$, $t_{i+3}=0$.
In the first case we have $t_{i+1}t_{i+2}=10$, which is abelian equivalent to $01$, and hence we have the semi-abelian return $t_{i}=0$. 
In the second case $t_{i+1}t_{i+2}=11$, which is not abelian equivalent to $01$, so we consider the next factor $t_{i+2}t_{i+3}=10\sim_{ab}01$, which gives the semi-abelian return $t_{i}t_{i+1}=01$. 
Symmetrically, $10$ gives semi-abelian returns $1$ and $10$. So the abelian class of $01$ has four semi-abelian returns: $\{0, 1, 01, 10\}$ and three abelian returns since $01 \sim_{ab}10.$ \end{example}

For our considerations we will use the following definitions. We
say that a letter $a$ is \emph{isolated} in a word
$w\in\Sigma^{\omega}$, if $aa$ is not a factor of $w$. A letter
$a\in\Sigma$ \emph{appears in} $w$ \emph{in a block of length}
$k>0$, if a word $b a^k c$ is factor of $w$ for some letters
$b\neq a$, $c\neq a$.

\medskip

In this paper we establish a new characterization of Sturmian
words analogous to Theorem \ref {returns}. Namely, we prove that a
recurrent infinite word is Sturmian
 if and only if each of its factors has two or three abelian returns (see Theorem \ref{main} in the Introduction).
 On the other hand, contrary to property of being Sturmian,
 abelian returns do not give a simple characterization of periodicity analogous to Proposition \ref {per}.
 In terms of semi-abelian returns Sturmian words have
exactly the same characterization as in terms of abelian returns
(see Theorem \ref{semi-abelian} in Introduction).

\section{Abelian returns and periodicity}

In this section we discuss relations between periodicity and
numbers of abelian and semi-abelian returns. We begin by proving a
simple sufficient condition for periodicity:

\begin{lemma} \label {periodic} Let $|\Sigma|=k.$ If each factor of a recurrent infinite word over the alphabet $\Sigma$
has at most $k$ abelian returns, then the word is
 periodic. \end{lemma}

\noindent\emph{Proof.}  Let $w$ be a recurrent word over a $k$-letter alphabet, and let
$v$ be a factor of $w$ containing
all letters from the alphabet. Consider two occurrences of $v$ in $w$,
say in positions $m$ and $n$ (with $m<n$). Then the abelian class of $w_{m}\dots w_{n-1}$
has all letters as abelian returns, and hence no more, because every factor of $w$
must have at most $k$ abelian returns.
Thus $w$ is periodic with period $n-m$. \qed

\medskip

\noindent \textbf{Remark.} Actually, this proves something
stronger: Let $w$ be any aperiodic word over an alphabet $\Sigma$,
$|\Sigma|=k$, and let $u$ be any factor of $w$ containing $k$
distinct letters, and let $vu$ be any factor of $w$ distinct from
$u$ beginning in $u$. Then the abelian class of $v$ must have at
least $k$ abelian returns. It follows that if a word is not
periodic, then for every positive integer $N$ there exists an
abelian factor of length $>N$ having at least $k+1$ abelian
returns. In other words, the value $k+1$ must be assumed
infinitely often.

\medskip

\noindent \textbf{Remark.} Notice that the condition given by Lemma \ref {periodic} is not necessary for periodicity. It is not difficult to construct a periodic word such that some of its factors have more than $k$ abelian returns.

\bigskip

Notice also that a characterization of periodicity similar to Proposition \ref {per} in terms of abelian returns does not exist. Moreover, in the case of abelian returns it does not hold in both directions. Consider an infinite aperiodic word of the form $\{110010, 110100\}^{\omega}$. It is easy to see that the factor $11$ has one abelian return $110010\sim_{ab}110100$. 
So, the existence of a factor having one abelian return does not
guarantee periodicity. The converse is not true as well: there
exist  periodic words such that each factor has at least
two abelian returns. An example is given by the following word
with period 24:
\begin{equation}\label{24}
w=(001101001011001100110011)^{\omega}.\end{equation} To
check that every factor of this word has at least two abelian
returns, one can check the factors up to the length $12$. If we
denote the period of $w$ by $u$, then every factor $v$ of length
$12<l\leq 24$ has the same abelian returns as abelian class of
words of length $24-l$ obtained from $u$ by deleting $v$. For a
factor of length longer than $24$ its abelian returns coincide
with abelian returns of part of this factor obtained by shortening
it by $u$.

\bigskip

Now we continue with relations between semi-abelian returns and
periodicity. In this connection semi-abelian returns show
intermediate properties between normal and abelian returns. E.~g.,
normal returns admit the characterization of periodicity given by
Proposition \ref {per}, for abelian returns the proposition does
not hold in both directions, and in the case of semi-abelian
returns the proposition holds in one direction giving a
sufficiency condition for periodicity:

\begin{proposition}
If a recurrent infinite word has a factor with one semi-abelian
return, then the word is periodic.
\end{proposition}

\noindent\emph{Proof.} It is readily verified that this unique
semi-abelian return word gives the period. \qed

\bigskip

We note that this condition is not necessary for periodicity. One
can take the same example \eqref{24} of a periodic word as for
abelian returns. Since each of its factors has at least two abelian
returns, it has at least two semi-abelian returns.

Lemma \ref {periodic} holds also for semi-abelian returns (exactly
the same proof works):

\begin{lemma}
Let $|\Sigma|=k.$ If each factor of a recurrent infinite word over
the alphabet $\Sigma$ has at most $k$ semi-abelian returns, then
the word is periodic.
\end{lemma}

\section{The structure of abelian returns of Sturmian words}

In this section we prove the ``only if'' part of Theorem
\ref{main}, and in addition we establish some properties concerning the
structure of abelian returns of Sturmian words.

 The following proposition
follows directly from definitions and basic properties of Sturmian
words:

\begin{proposition} \label{aBb} Semi-abelian returns of factors of a Sturmian word are Christoffel words.
\end{proposition}

\noindent \emph{Proof.} Consider semi-abelian return to a factor
$v$ of length $n$ starting at position $i$ of a Sturmian word $w$.
We should prove that its semi-abelian return is either a letter or
of the form $aBb$, where $a\neq b$ are letters, $B$ is a bispecial
factor of $w$. If $w_i=w_{i+n}$, then the letter $w_i$ is
semi-abelian return. If $w_i=a$, $w_{i+n}=b$, $a\neq b$, then
there exists $k\geq 0$, such that $w_{i+1}\dots w_{i+k} =
w_{i+1+n}\dots w_{i+k+n}$, and $w_{i+k+1}\neq w_{i+k+1+n}$. Since
$w$ is balanced, we have that $w_{i+k+1}=b$, $w_{i+k+1+n}=a$. So,
$w_{i+k+2}\dots w_{i+k+n+1}\sim_{ab}v$, and $w_{i}\dots
w_{i+k+1}\sim_{ab} w_{i+n}\dots w_{i+k+n+1}$ is semi-abelian
return to $v$. By definition the factor $w_{i+1}\dots w_{i+k} =
w_{i+1+n}\dots w_{i+k+n}$ is bispecial. \qed

\begin{corollary} \label{1}  Fix $l\geq 2.$ Then each factor  $u$ of a Sturmian word  has at most one abelian return of length $l$. \end{corollary}

\smallskip

Now we proceed to the "only if" part of Theorem \ref{main}:

\begin{proposition} \label {necessity} Each factor of a Sturmian word has two or three abelian returns.
\end{proposition}

The proof of this proposition is based on the characterization of
balanced words presented in \cite {jz}. We will need some notation
from the paper.

Suppose $1\leq p <q$ are positive integers such that $\gcd(p,
q)=1$. Let $\mathscr{W}_{p, q}$ denote the set of all words $w \in
\{0, 1\}^q$ with $|w|_1 =p$. If $w\in \mathscr{W}_{p, q}$ then the
symbol $1$ occurs with frequency $p/q$ in $w$. Define the
\emph{shift} $\sigma: \{0, 1\}^\omega \to \{0, 1\}^\omega$ by
$\sigma(w)_i=w_{i+1}$. Similarly define $\sigma: \{0, 1\}^q \to
\{0, 1\}^q$ by $\sigma(w_0\dots w_{q-1}) = w_1\dots w_{q-1}w_0$.

Since $\gcd(p, q)=1,$ it follows that any element of $\mathscr{W}_{p, q}$ has
the least period $q$ under the shift map $\sigma$. We will write
$w\sim w'$ if there exists $0\leq k \leq q-1$ such that $w'
=\sigma^k(w)$. In this case we say that $w$, $w'$ are
\emph{cyclically conjugate}, or that $w$, $w'$ are cyclic shifts
of one another. The equivalence class $\{\sigma^i(w): 0\leq i <
q\}$ of each $w\in\mathscr{W}_{p, q}$ contains exactly $q$
elements. Let
$$\mathbb{W}_{p,q} = \mathscr{W}_{p, q} / \sim$$
denote the corresponding quotient. Elements of $\mathbb{W}_{p,q}$
are called orbits. It will usually be convenient to denote an
equivalence class in $\mathbb{W}_{p,q}$ by one of its elements
$w$.

Given an orbit $[w]\in \mathbb{W}_{p,q}$, let
$$w_{(0)}<_L w_{(1)} <_L \dots <_L w_{(q-1)}$$
denote the lexicographic ordering of its elements. Define the
lexicographic array $A[w]$ of the orbit $[w]$ to be the $q\times
q$ matrix whose $i$th row is $w_{(i)}$. We will index this array
by $0\leq i, j \leq q-1$, so that $A[w] =
(A[w]_{ij})_{i,j=0}^{q-1}$. For $0\leq i, j \leq q-1$,  let
$w_{(i)}[j]$ denote the length-$(j+1)$ prefix of $w_{(i)}$; so the
$w_{(i)}[j]$ are the length-$(j + 1)$ factors of $w$, counted with
multiplicity. For each $j$ this induces the following
lexicographic ordering:
$$w_{(0)}[j] \leq_L w_{(1)}[j] \leq_L \dots \leq_L w_{(q-1)}[ j].$$

\begin{theorem} {\rm \cite{jz}} Suppose $w \in \{0,1\}^q$. The following are equivalent:

\noindent {\rm (1)} $w$ is a balanced word,

\noindent {\rm (2)} $|w_{(i)}[j]|_1\leq |w_{(i+1)}[j]|_1$ for all $0\leq
i\leq q - 2$ and $0\leq j\leq q - 1$.

\end{theorem}

The following proposition from \cite{jz} gives a very practical
way of writing down the lexicographic array associated to a
balanced word.

\begin{proposition} {\rm \cite{jz}} Let $[w]$ be the unique balanced orbit in $\mathbb{W}_{p,q}$. Define $u\in \mathscr{W}_{p, q}$ by
$$u = 0 \dots 0 \underbrace{1 \dots 1}_p$$
Then, for $0\leq i, j \leq q - 1$,

\noindent {\rm (1)} $A[w]_{ij} =(\sigma^{jp}u)_i$,

\noindent {\rm (2)} The $j$th column of $A[w]$ is (the vector
transpose of) the word $\sigma^{jp}u$

\noindent {\rm (3)} $w_{(i)} =u_i(\sigma^{p}u)_i(\sigma^{2p}u)_i
\dots (\sigma^{(q-1)p}u)_i$.

\end{proposition}

\begin{example}\label{example2}  Consider a balanced word $w=0101001\in
\mathscr{W}_{p, q}$. The lexicographic ordering of $[w]$ is
$$0010101 <_L 0100101 <_L 0101001 <_L 0101010 <_L 1001010 <_L 1010010 <_L 1010100,$$
so the corresponding lexicographic array is
$$A[w]= \left( \begin{array}{ccccccc}
0 & 0 & 1 & 0 & 1 & 0 & 1 \\
0 & 1 & 0 & 0 & 1 & 0 & 1 \\
0 & 1 & 0 & 1 & 0 & 0 & 1 \\
0 & 1 & 0 & 1 & 0 & 1 & 0 \\
1 & 0 & 0 & 1 & 0 & 1 & 0 \\
1 & 0 & 1 & 0 & 0 & 1 & 0 \\
1 & 0 & 1 & 0 & 1 & 0 & 0 \end{array} \right) $$ \end{example}

We now apply the above technique for studying abelian
returns as follows:

Fix a Sturmian word $s$ and a factor $v.$ First notice that $v$
cannot have only one abelian return, otherwise we immediately get a
contradiction with the irrationality of letter frequencies in $s$.
We consider a standard factor $w$ 
of $s$ of long enough length to contain $v$ and all abelian
returns to $v$. Let $|w|=q$, $|w|_1=p$. Then all the conjugates of
$w$ are factors of $s$, they are pairwise distinct, and
$\gcd(p,q)=1$ (see, e. g. \cite{mr}). Without loss of generality we can assume
that $v$ is "poor"\ in $1$-s, i.e., it contains fewer
 $1$'s than the unique other abelian class of the same
length. Then if we consider in $A[w]$ the words $w_{(i)}[j]$, we
have that there exists $n< q-1$ such that
$w_{(i)}[j]\sim_{ab}v$ for $0\leq i \leq n$, and
$w_{(i)}[j]\not\sim_{ab}v$ for $n < i \leq q-1$. Note also that
$A[w]_{im}=A[w]_{(i+q-p)(m+1)}$; from now on the indices are taken
modulo $q$.

The lexicographic array allows us to find abelian returns to $v$
as follows: For a word $u$ denote by $u[m, l]$ the factor $u_m
\dots u_l$. If for an $i$, $0\leq i \leq n$, we have $w_{(i)}[k,
k+j]\sim_{ab}v$, where $v$ is as above and $k>0$ is the minimal
such length, then by definition $w_{(i)}[k-1]$ is a semi-abelian
return to $v$. Notice also that if $A[w]_{(i-1)k}=1$ and
$A[w]_{ik}=0$, then $w_{(m)}[k, k+j]\sim_{ab}v$ for $m=i, \dots,
i+n$. That is, we have exactly $n+1$ words from the abelian class
of $v$ starting in every column, and these words are in
consecutive $n+1$ rows (the first and the last row are considered
as consecutive).

\begin{example}\label{example3} Consider abelian returns to the abelian
class of $001$ in the Example \ref{example2}. $w_{(i)}[2]\sim_{ab}001$
for $0\leq i \leq 4$; $w_{(i)}[1,3]\sim_{ab}001$ for $i= 4, 5,
6, 0, 1$, $w_{(i)}[2,4]\sim_{ab}001$ for $i= 1, \dots, 5$. So,
the abelian returns are $w_{(0)}[0]=w_{(1)}[0]=0$, $w_{(4)}[0]=1$,
$w_{(2)}[1]=w_{(3)}[1]=01$. \end{example}

\noindent \emph{Proof of Proposition \ref{necessity}.} Suppose
that some factor $v$ of length $j+1$ has at least $4$ abelian
returns. Without loss of generality we may assume that $v$ is poor
in $1$, and in the lexicographic array, rows $0\dots n$ start with
factors from the abelian class of $v$. By Corollary \ref {1} there
can be at most one abelian return of a fixed length greater than
$1$ (length $1$ will be considered separately), so in a
lexicographic array we must have one of the following situations:

\medskip

\noindent 1) there exist $k_1<k_2$ and $n_1<n_2<n$ such that
$w_i[j]$ has semi-abelian returns of length $k_1$ for $i=1,\dots,
n_1$, $w_i[j]$ has semi-abelian returns of length $k_2$ for
$i=n_1+1,\dots, n_2$, and $w_{n_2+1}[j]$ has semi-abelian returns
of length greater than $k_2$;

\smallskip

\noindent 2) symmetric case: there exist $k_1<k_2$ and $n_1<n_2<n$
such that $w_i[j]$ has semi-abelian returns of length $k_2$ for
$i=n_1+1,\dots , n_2$, $w_i[j]$ has semi-abelian returns of length
$k_1$ for $i=n_2+1, \dots , n$, and $w_{n_1}[j]$ has semi-abelian
returns of length greater than $k_2$.

\medskip

We consider only case 1) as the proof of case 2) is similar.
First, in case 1) one can notice that the words $w_{n_1}[k_1,
k_1+q]$ and $w_{n_2}[k_2, k_2+q]$ coincide. So if we consider
semi-abelian returns "to the left" of the words $w_{n_1}[k_1,
k_1+j]$ and $w_{n_2}[k_2, k_2+j]$, they should be the same, but
they are not: the first one is of length $k_1$, the second one is
of length $k_2$.

It remains to consider the case when $v$ has both letters as
abelian returns. It can be seen directly from the lexicographic
array, that the third and the last return is $01$ (in this case
after a word not from abelian class of $v$ we will necessarily
have a word from abelian class of $v$, i.e., the longest possible
length of abelian return is $2$). \qed

\begin{theorem}\label{singular} A factor of a Sturmian word has two abelian returns if and only if it is singular.
\end{theorem}

\noindent \emph{Proof.} The method of proof is similar to the
proof of Proposition \ref{necessity} and relies upon the
characterization of balanced words from \cite {jz}.

If a factor is singular, then it is the only word in its abelian
class, so its semi-abelian returns coincide with usual returns.
Since every factor of a Sturmian word has two returns
\cite{vuillon}, then a singular factor has two semi-abelian
returns, and hence two abelian returns.

Now we will prove the converse, i.e., that if a factor $v$ of a Sturmian word $s$
of length $j+1$  has two abelian returns, then it
is singular.

As in the proof of Proposition \ref {necessity}, we consider a
standard factor $w$ of $s$ of long enough length to contain $v$
and all abelian returns to $v$, and denote $|w|=q$, $|w|_1=p$.
Without loss of generality we again assume that $v$ is "poor"\ in
$1$-s, so that there exists $n< q-1$ such that
$w_{(i)}[j]\sim_{ab}v$ for $0\leq i \leq n$, and
$w_{(i)}[j]\not\sim_{ab}v$ for $n < i \leq q-1$.

\medskip

It is not difficult to see that two abelian returns are possible
in one of the following cases:

\medskip

\noindent Case 1) there exist $0 \leq m < n$, $0 < k_1,k_2 < q$
such that $w_{(i)}[k_1-1]$ is semi-abelian return for all $0\leq
i\leq m$, $w_{(i)}[k_2-1]$ is semi-abelian return for all $m+1\leq
i \leq n$;

\medskip

\noindent Case 2) there exist $0 \leq m_1 < m_2 < n$, $0< k_1 <
k_2 < q$ such that $w_{(i)}[k_1-1]$ is semi-abelian return for all
$0\leq i\leq m_1$ and $m_2+1\leq i\leq n$; $w_{(i)}[k_2-1]$ is
semi-abelian return for all $m_1+1\leq i \leq m_2$.

\bigskip

\noindent \textbf{Case 1)} In case 1) we will assume that
$k_1< k_2$, the proof in case $k_2< k_1$ is symmetric. We will
consider two subcases:

\medskip

\noindent \textbf{Case 1a)} $A[w]_{m k_2}=1$, $A[w]_{(m+1)
k_2}=0$. This means that $w_{(i)}[k_2, k_2+j]\sim_{ab}v$ for
$i=m+1,\dots, m+n+1$, and $A[w]_{m (k_2-1)}=0$, $A[w]_{(m+1)
(k_2-1)}=1$. 
So, the element $A[w]_{(m+1) k_2}$ is a left-upper element of a
block of abelian class of $v$, and $A[w]_{m (k_2-1)}$ is a
right-lower element of another such block. It is easy to see that
the latter block starts in column $k_1$. Therefore,
$|v|=j+1=k_2-k_1<k_2$.

In case 1a) we will prove that the abelian class of $v$
consists of a single word, i.e., $w_{(i)}[j]=v$ for $i=0,\dots,
n$. Suppose that $w_{(i)}[j]\neq w_{(i+1)}[j]$ for some $i\in
\{0,\dots, n-1\}$. Since the rows grow lexicogaphically, it means
that there exists $0\leq l < j<k_2-1$ such that $A[w]_{il}=0$,
$A[w]_{(i+1)l}=1$. Hence $A[w]_{i(l+1)}=1$, $A[w]_{(i+1)(l+1)}=0$,
and so $w_{(i+1)}[l+1,l+1+j]\sim_{ab} v$. If $m<i+1\leq n$,
then the word $w_{(i+1)}[j]$ has return  $w_{(i+1)}[l]$, which is
impossible, because it has return $w_{(i)}[k_2]$. Similarly we get
that the case  $0\leq i+1\leq m$ and $l+1<k_1$ is impossible.

In  case $0\leq i+1\leq m$ and $k_1\leq l+1 <k_2$ we get that the
word $w_{(i+1)}[k_1,k_1+j]$ has return $w_{(i+1)}[k_1,l]$ of
length $l-k_1+1$. But in this case $w_{(t)}[l+1, l+1+j]\sim_{ab}v$
for $t=i+1,\dots ,i+1+n$. Contradiction with the condition that
$w_{(t)}[k_2-1]$ is semi-abelian return to $w_{(t)}[j]$. So, the
case $0\leq i+1\leq m$ and $k_1\leq l+1 <k_2$ is impossible. Hence
$w_{(i)}[j] = w_{(i+1)}[j]$ for $i=0,\dots, n-1$, i.e., the
abelian class of $v$ consists of a single word.

\smallskip

\noindent \textbf{Case 1b)} $A[w]_{m k_2}= 0$ or $A[w]_{(m+1)
k_2}=1$. This means that $w_{(m)}[k_2, k_2+j]\sim_{ab}v$. Hence
the word $w_{(n)}[j]$ has semi-abelian return $w_{(n)}[k_2]$ of
length $k_2+1$, and the word $w_{(m)}[k_1, k_1+j]$ has
semi-abelian return $w_{(m)}[k_1,k_2]$ of length $k_2-k_1+1$, so
the returns are different. This is impossible since
$w_{(n)}=w_{(m)}{[k_1, k_1+q-1]}$.

\bigskip

\noindent \textbf{Case 2)} In case 2) the fact that $w_{(i)}[k_1]$
is semi-abelian return for all $0\leq i\leq m_1-1$ and $m_2+1\leq
i\leq n$ implies that $n>q/2$. So, $k_1=1$, i.e., we necessarily
have return(s) of length $1$. Since there are two abelian returns
totally, we can have only one return of length $1$, and this
return is $0$. It means that $A[w]_{i0}=0$ for $0\leq i\leq n$.
Since $w_{(m_2)}[1,j+1]\not\sim_{ab} v$ and
$w_{(m_2+1)}[1,j+1]\sim_{ab} v$, we have $A[w]_{m_2 1}=1$,
$A[w]_{(m_2+1)1}=0$, and hence $A[w]_{m_2 0}=0$,
$A[w]_{(m_2+1)0}=1$. We get a contradiction with $A[w]_{i0}=0$ for
$0\leq i\leq n$.

\medskip

So, the converse is proved, i.e., every factor of a Sturmian word
having two abelian returns is singular. \qed

\section{Proof of Theorem \ref {main}: the sufficiency}

Here we prove the "if" part of Theorem \ref {main}, i.e., we
establish the condition on the number of abelian returns forcing a
word to be Sturmian, i.e., we prove that a binary recurrent word with each factor having two or three abelian returns is Sturmian.

\begin{proposition} \label {sufficiency}
If each factor of a binary recurrent infinite word has at most
three abelian returns and at least two semi-abelian returns, then the word is balanced.
\end{proposition}

Notice that we formulate and prove auxiliary lemmas and propositions in a bit stronger way than we need for sufficiency in Theorem \ref {main}: instead the condition ``each factor has two or three abelian returns'' we put a weaker condition ``each factor has at most three abelian and at least two semi-abelian returns''. Using this condition we will be able to prove the sufficiency in both Theorems \ref{main} and \ref{semi-abelian}: since both words with two or three abelian returns and words with two or three semi-abelian returns satisfy this condition, we solve two problems at once.

The proof of this proposition is rather technical, it is based on
considering abelian returns to different possible factors of the
infinite word and consecutively restricting the possible form of
the word. Denote the binary word with at most three abelian
returns by $w\in \{0,1\}^{\omega}$. 
In the rest of
this section instead of abelian returns "to the left" we consider
abelian returns "to the right": if $vu$ is a factor having
$v'\sim_{ab}v$ as its suffix, and $vu$ does not contain as factors
other words abelian equivalent to $v$ besides suffix and prefix,
then the abelian class of $u$ is abelian return to the abelian
class of $v$. It is easy to see that regardless of the definition,
the set of abelian returns to each abelian factor is the same. We
will refer to the word $u$ as \emph{right semi-abelian return} of
the abelian class of $v$, so normal semi-abelian returns can be
regarded as left semi-abelian returns. Right semi-abelian returns
do not necessarily coincide with left semi-abelian returns, but
their abelian classes also give the set of abelian returns. Though
this does not make any essential difference, this modification of
the definition is more convenient for our proof of this
proposition.

\medskip

We will
make use of the following key lemma:

\begin{lemma} \label {isolated}
If each factor of a binary recurrent infinite word $w$ has at most
three abelian and at least two semi-abelian
returns, then one of the letters is isolated.
\end{lemma}

\noindent\emph{Proof.} 
Considering abelian returns to letters, we get that every
letter can appear in blocks of at most three different lengths.
Denote these lengths for blocks of $0$'s by $l_1$, $l_2$, $l_3$,
where $l_1<l_2<l_3$, for blocks of $1$'s by $j_1$, $j_2$, $j_3$,
where $j_1<j_2<j_3$. Notice that a letter can appear in blocks of
only two or one lengths, then the third length or the third and
the second lengths are missing.

Consider right semi-abelian returns of the word $1 0^{l_1}$: they
are $1$, $0^{l-l_1} 1$ for $l=l_2$, $l_3$ (if $0$ appears in
blocks of corresponding lengths), $1^{j-1} 0^{l_1}$ for $j=j_1>1,
j_2, j_3$ (if $1$ appears in blocks of corresponding lengths) and
$0$ for $j_1=1$. Some of these returns should be missing or
abelian equivalent to others in order to have at most three
abelian returns totally. So we have the following cases:

\smallskip

\noindent -- $j_2$, $j_3$, $l_3$ are missing, i.e., $w \in \{
0^{l_1}1^{j_1}, 0^{l_2}1^{j_1} \}^{\omega}$. In this case abelian
returns are $1$, $0^{l_2-l_1} 1$, and  $1^{j_1-1}0^{l_1}$ for
$j_1>1$ or $0$ for $j_1=1$.

\noindent -- $l_2$, $l_3$, $j_3$ are missing, i.e., $w \in \{
0^{l_1}1^{j_1}, 0^{l_1}1^{j_2} \}^{\omega}$. Abelian returns are
$1$, $1^{j_2-1}0^{l_1}$, and $1^{j_1-1}0^{l_1}$, if $j_1>1$, or
$0$, if $j_1=1$.

\noindent -- $j_2$, $j_3$ are missing, $j_1=2$, $l_2=2l_1$ or
$l_3=2l_1$, i.e., $w \in ( \{ 0^{l_1}, 0^{2l_1}, 0^{l}\} 1^{j_2}
)^{\omega}$. Abelian returns are $1$, $0^{l_1} 1$, $0^{l-l_1} 1$.

\noindent -- $l_3$, $j_3$ are missing, $l_2=2l_1$, $j_1=2$ or
$j_2=2$, $w \in ( \{ 0^{l_1}, 0^{2l_1}\} \{1^{2}, 1^{j}\})^{\omega}$. Abelian returns are $1$,
$0^{l_1} 1$, $1^{j-1}0^{l_1}$ (if $j>1$) or $0$ (if $j=1$).

\noindent -- $j_2$, $l_2$, $j_3$, $l_3$ are missing, then $w = (
0^{l_1} 1^{j} )^{\omega}$ is periodic. This case is impossible
since $0^{l_1}$ has only one semi-abelian return.

\smallskip

Notice that the first two cases are symmetric. Considering abelian
returns to the word $1^{j_1}0$, we get symmetric cases ($0$ change
places with $1$, $j_k$ change places with $l_k$, $k=1,2,3$).
Combining the cases obtained by considering abelian returns to $1 0^{l_1}$ with the cases obtained by considering abelian returns to $1^{j_1}0$,
we finally get the following remaining cases (up to renaming letters):

\medskip

\noindent 1) $j_2$, $j_3$, $l_3$ are missing, i.e. $w$ is of the
form  $w\in \{ 0^{l_1}1^{j_1}, 0^{l_2}1^{j_1}\}^{\omega}$.

\smallskip

\noindent 2) $l_3$, $j_3$ are missing, $l_1=1$, $l_2=2$, $j_1=2$,
$j_2=4$, i.e. $w\in (\{ 0, 0^2\}\{1^2, 1^4\})^{\omega}$.
\smallskip

\noindent 3) $l_3$, $j_3$ are missing, $l_1=1$, $l_2=2$, $j_1=1$,
$j_2=2$, i.e. $w\in (\{ 0, 0^2\}\{1, 1^2\})^{\omega}$.

\smallskip

\noindent 4) $l_3$, $j_3$ are missing, $l_1=2$, $l_2=4$, $j_1=2$,
$j_2=4$.\, i.e.
 $w\in (\{ 0^2, 0^4\}\{1^2, 1^4\})^{\omega}$.

\medskip

\noindent \textbf{Case 1)}: $w\in \{ 0^{l_1}1^j_1,
0^{l_2}1^j_1\}^{\omega}$.

In the first case we should prove that $j_1=1$. We omit the index
$1$ for brevity: $j=j_1$. Suppose that $j>1$. Consider right
abelian returns to the word $1 0^{l_2}$. They are $1$,
$1^{j-1}(0^{l_1}1^j)^k 0^{l_2}$ for all $k\geq 0$ such that the
word $0^{l_2}1^j(0^{l_1}1^j)^k 0^{l_2}$ is a factor of $w$.
Therefore, we have at most two values of $k$ (probably, including
$0$).

Right abelian returns to the word $1^j 0^{l_1} 1$ are $1$,
$(0^{l_2}1^j )^m 0^{l_1} 1$ for all $m\geq 0$ such that the word
$10^{l_1}1^j(0^{l_2}1^j)^m 0^{l_1} 1$ is a factor of $w$. So, we
have at most two values of $m$ (probably, including $0$).

Notice that we cannot have only one value of $k$ and only one
value of $m$ simultaneously, since in this case we have periodic
word $w= ( (0^{l_1} 1^j)^{k_1} (0^{l_2}1^j)^{m_1} )^{\omega}$, and
the word $(0^{l_2}1^{1^j})^{m_1-1}0^{l_2}$ has only one
semi-abelian return.

Taking into account conditions for $m$ and $k$, which we have just
obtained from considering abelian returns to both $1 0^{l_2}$ and
$1^j 0^{l_1} 1$, we find that there are two opportunities:

\smallskip

\noindent \textbf{Case 1a)} $w\in ( \{ (0^{l_1}1^j)^{k_1},
(0^{l_1}1^j)^{k_2} \} 0^{l_2} 1^j)^{\omega}$, $0<k_1<k_2$.
The word $0^{l_2} 1^j 0^{l_1} 1^{j-1}$ has returns $1$,
$0^{l_1}1$, $0^{l_2} ( 1^j 0^{l_1} )^{k-1}1 $ for all $k$ such
that the word $0^{l_2}1^j(0^{l_1}1^j)^k 0^{l_2}$ is a factor of
$w$. To provide at most three abelian returns, $w$ should admit
only one value of $k$. In this case there is also only one value of $m$, so the case 1a) is impossible.

\smallskip

\noindent \textbf{Case 1b)} $w\in ( 0^{l_1} 1^j, \{ (0^{l_2}1^j)^{m_1},
(0^{l_2}1^j)^{m_2} \} )^{\omega}$, $0<m_1<m_2$.
The word $1^j 0^{l_1} 1^j 0^{l_2} 1$ has returns $1$,
$10^{l_2}$, $10^{l_1} ( 1^j 0^{l_2} )^{m-1} $ for all $m$ such
that the word $10^{l_1}1^j(0^{l_2}1^j)^m 0^{l_1}1$ is a factor of
$w$. To provide at most three abelian returns, $w$ should admit
only one value of $m$. In this case there is also only one value of $k$, so the case 1b) is impossible.

\smallskip

Thus, in case 1) $1$'s are isolated.

\bigskip

In cases 2)--4) we need to consider words containing all four
blocks, otherwise we get into conditions of case 1) in which we
proved that $1$-s are isolated. The proof is similar for the three
cases, and is based on studying abelian returns of certain type.
When we examine $w\in ( \{0^{l_1}, 0^{l_2} \}, \{1^{j_1}, 1^{j_2}
\})^{\omega}$, we consider abelian returns to the words $0^{l_1}
1^{j_2}$ and $0^{l_2} 1^{j_1}$, and with a technical case study
obtain that if both words have at most three abelian returns, then
$w$ is periodic of a certain form, and then find its factor having
one semi-abelian return.

\bigskip

\noindent \textbf{Case 2)}: $w\in (\{ 0^2, 0^4\}\{1,
1^2\})^{\omega}$.

Consider abelian returns of the word $0^2 1^2$. Factors of $w$
from the abelian class of $0^2 1^2$ are the following: $0^2 1^2$,
$1^2 0^2$, $0110$, $1001$. Notice that each of these words is
necessarily a factor of $w$. Consider right semi-abelian returns
to each factor:

\begin{itemize}

\item $0^2 1^2$, $01^20$ have right semi-abelian return $0$

\item $1^2 0^2$ has right semi-abelian returns of the form
$\alpha_1=(0^210^2)^{i_1} 1$ and/or $\alpha_2= (0^2 10^2)^{i_2}
0^2 1^2$ for some $i_1, i_2\geq 0$


\item $10^21$ has right semi-abelian returns of the form $\alpha_3=(0^4 1)^{i_3}
1$ and/or $\alpha_4=(0^41)^{i_4} 0^2 1$ for some $i_3, i_4\geq 0$

\end {itemize}

We will also use abelian returns of the word $0^4 1$: 

\begin{itemize}

\item  $0^4 1$ could have right semi-abelian returns $0$, returns of the
forms $\alpha'_1=(1 0^2 1)^{j_1} 0^2$ with $j_1>0$ and $\alpha'_2=
(1 0^2 1)^{j_2} 1 0^4$ for some $j_2\geq 0$

\item $0^3 1 0$, $0 1 0^3$  (not necessarily factors of $w$) have
right semi-abelian return $0$

\item $0^2 1 0^2$ could have right semi-abelian returns $0$, returns of the
forms $\alpha'_3=(1^2 0^2)^{j_3} 0^2$ with $j_3>0$ and $\alpha'_4=
(1^2 0^2)^{j_4} 1 0^2$ for some $j_4\geq 0$


\item $1 0^4$ has right semi-abelian return $1$.

\end{itemize}
\noindent These are summarized in the table below:

\medskip

\noindent \begin{tabular}{|c|l|l|} \hline
                 abelian class & word  & possible right semi-abelian returns  \\ \hline

 \multirow{3}{*}{$0^2 1^2$} & $0^2 1^2$, $01^20$ & $0$

\\ \cline{2-3}
& $1^2 0^2$ & $\alpha_1=(0^210^2)^{i_1} 1$, $\alpha_2= (0^2
10^2)^{i_2} 0^2 1^2$ for some $i_1, i_2\geq 0$
\\ \cline{2-3}
   &  $10^2 1$ & $\alpha_3=(0^4
1)^{i_3} 1$, $\alpha_4=(0^41)^{i_4} 0^2 1$ for some $i_3, i_4\geq
0$
\\ \hline \hline

 \multirow{3}{*}{$0^4 1$} & $0^4 1$ & $0$, $\alpha'_1=(1 0^2 1)^{j_1} 0^2$ with $j_1>0$, $\alpha'_2=
(1 0^2 1)^{j_2} 1 0^4$ for some $j_2\geq 0$

\\ \cline{2-3}
& $0^3 1 0$, $0 1 0^3$ & $0$
\\ \cline{2-3}
   & $0^2 1 0^2$ &  $0$, $\alpha'_3=(1^2 0^2)^{j_3} 0^2$ with $j_3>0$,
$\alpha'_4= (1^2 0^2)^{j_4} 1 0^2$ for some $j_4\geq 0$
\\ \cline{2-3}
& $10^4$ & $1$

\\ \hline

\end{tabular}

\medskip

\medskip

Notice that $\alpha_1 \sim_{ab} \alpha_3$ when $i_1=i_3$, and
$\alpha'_1 \sim_{ab} \alpha'_3$ when $j_1=j_3$.

\smallskip

If factors from the abelian class of $0^2 1^2$ have only letters
as abelian returns, then we obtain periodic word $w=(0^2
1^2)^{\omega}$, and this word does not contain all four blocks.
So, a factor from the abelian class of $0^2 1^2$ should contain an
abelian return of length longer than $1$ (referred to as
\emph{long returns} in the further text), so we consider the four
cases corresponding to returns $\alpha_1$ through $\alpha_4$.

\smallskip

\noindent \textbf{Case 2a)} let $1^2 0^2$ have a return $\alpha_1$
with $i_1>0$. Then $w$ contains a factor $u=1^2 0^2
(0^210^2)^{i_1} 1$. Now consider right semi-abelian returns to the
abelian class of $0^4 1$. One can find right semi-abelian returns
$0$ (in the factor $0^4 1 0$ of $u$) and $1$ (in $10^41$). Since
$u$ has a prefix $1^2 0^4$, it means that there is a long right
semi-abelian return ending in $1^2 0^4$, i.e., we have right
semi-abelian return $\alpha'_2$ or $\alpha'_3$. A suffix $0^210^2
1 $ of $u$ implies that there is a long right semi-abelian return
$\alpha'_3$ or $\alpha'_4$. So, the only possibility is that an
abelian class of $0^4 1$ has abelian returns $0$, $1$ and
$\alpha'_3 \sim_{ab} \alpha'_1$ with $j_1=j_3>0$, and hence
nothing else. The factor $u$ has a suffix $0^210^2 1$, so the
factor $0^210^2$ here has right semi-abelian return $\alpha'_3$,
and therefore $u$ is continued in the unique way: $u'=1^2 0^2
(0^210^2)^{i_1} (1^2 0^2)^{j_3} 0^2$. One can find here two right
semi-abelian returns $0$ and $1$ to the abelian class of $0^2
1^2$, and we started with the first long right semi-abelian return
$\alpha_1$, so the three returns to $0^2 1^2$ are $0$, $1$ and
$\alpha_1\sim_{ab} \alpha_3$. The factor $u'$ has a suffix $1^2
0^4$, so the factor $1^2 0^2$ here has right semi-abelian return
$\alpha_1$, therefore it is continued in the unique way: $u''=1^2
0^2 (0^210^2)^{i_1} (1^2 0^2)^{j_3} (0^210^2)^{i_1} 1$. Continuing
this line of reasoning, we obtain a periodic word. One can find a
factor having one semi-abelian return, e. g., $ (1^2 0^2)^{j_3-1}
1^2$. Hence $1^2 0^2$ has no long right semi-abelian returns of
the form $\alpha_1$.

\smallskip

\noindent \textbf{Case 2b)} let $1^2 0^2$ have a return $\alpha_3$
with $i_3>0$. Then $w$ contains a factor $u=1 0^2 1 (0^4 1)^{i_3}
1$. Now consider right semi-abelian returns to the abelian class
of $0^4 1$. One can find right semi-abelian returns $0$ (in the
factor $0^4 1 0$ of $u$) and $1$ (in $10^41$). Since $u$ has a
prefix $1 0^2 1 0^2$, it means that there is a long right
semi-abelian return ending in $1 0^2 1 0^2$, i.e., we have right
semi-abelian return $\alpha'_1$ or $\alpha'_4$. A suffix $0^4 1^2$
of $u$ implies that there is a long right semi-abelian return
$\alpha'_1$ or $\alpha'_2$. So, the only possibility is that an
abelian class of $1 0^4$ has abelian returns $0$, $1$ and
$\alpha'_1 \sim_{ab} \alpha'_3$ with $j_1=j_3>0$. The factor $u$
has a suffix $0^4 1^2$, so the factor $0^41$ here has right
semi-abelian return $\alpha'_1$, so $u$ is continued in the unique
way: $u'=1 0^2 1 (0^4 1)^{i_3} (1 0^2 1)^{j_1} 0^2$.
 This factor has a suffix $1 0^2 1 0^2$, so the factor $1 0^2 1$
 here has right semi-abelian return $\alpha_3$,
and therefore it is continued in the unique way: $u''=1 0^2 1 (0^4
1)^{i_3} (1 0^2 1)^{j_1} (0^4 1)^{i_3} 1$. Continuing this line of
reasoning, we obtain a periodic word. One can find a factor having
one semi-abelian return, e. g., $ (0^4 1)^{i_3-1} 0^4$.  Hence $1
0^2 1$ has no long right semi-abelian returns of the form
$\alpha_3$.

\smallskip

\noindent \textbf{Case 2c)} let $1^2 0^2$ have a return $\alpha_2$
with $i_2\geq 0$. Notice that if $1^2 0^2$ has only return $\alpha_2$, then $w=(1^2 0^2 (0^2 1 0^2)^{i_2} 0^2)^{\omega}$, and $w$ does not contain the block $0^2$. We proved that there is no long right semi-returns of the
forms $\alpha_1$ and $\alpha_3$, so the only possibility is that
$1^20^2$ has two returns $\alpha_2$ and $1$, and $10^21$ always
has return $1$, otherwise this abelian class has more than $3$
abelian returns. So, $1^20^2$ is followed by either $(0^2
10^2)^{i_2} 0^2 1^2 $ or $1$. In both cases we can determine
several next letters: in the first case the next symbols are $00$
(because $w$ contains maximum two consecutive $1$-s), in the
second case the next symbols are $100$ (since $10^21$ always has
return $1$, and $11$ is always followed by $00$). So, $1^20^2$ is
followed by either $(0^2 10^2)^{i_2} 0^2 1^2 0^2$ or $1^20^2$.
Both continuations have suffix $1^20^2$, which is followed by
either $1$ or $\alpha_2$, etc:

\begin{picture}(300,90)

\put(0,35){\makebox(50,10)[l]{$1^2 0^2$}}
\put(25,40){\line(1,1){20}} \put(25,40){\line(1,-1){20}}
\multiput(5,7)(89,10){2}{\put(45,50){\makebox(70,10)[l]{$(0^2 1
0^2)^{i_2} 0^2 1^2 0^2$}} \put(120,55){\line(1,1){10}}
\put(120,55){\line(1,-1){10}}}
\multiput(5,-7)(89,32){2}{\put(45,20){\makebox(90,10)[l]{$1^2
0^2$}} \put(68,25){\line(1,1){10}} \put(68,25){\line(1,-1){10}}}

\put(45,-17){\put(45,20){\makebox(90,10)[l]{$1^2 0^2$}}
\put(68,25){\line(1,1){10}} \put(68,25){\line(1,-1){10}}}

\put(45,-27){\put(45,50){\makebox(70,10)[l]{$(0^2 1 0^2)^{i_2} 0^2
1^2 0^2$}} \put(121,55){\line(1,1){10}}
\put(121,55){\line(1,-1){10}}}

\put(240,65){\makebox(20,10)[l]{$\dots$}}
\put(180,45){\makebox(20,10)[l]{$\dots$}}
\put(185,23){\makebox(20,10)[l]{$\dots$}}
\put(130,3){\makebox(20,10)[l]{$\dots$}}
\end{picture}

\medskip

\noindent Thus $w\in \{ (0^2 10^2)^{i_2} 0^2 1^2 0^2, 1^20^2
\}^{\omega}$. We are interested in the case when all four blocks
are contained in $w$, so we get $i_2>0$, otherwise $w$ does not
contain the block $1^1$, and we get into case 1), which we proved
is impossible.

So, $w$ contains a factor $u=1^2 0^2 (0^210^2)^{i_2} 0^2 1^2$ for
some $i_2>0$. Now consider abelian returns to the abelian class of
$0^4 1$. One can find right semi-abelian returns $0$ (in the
factor $0^4 1 0$ of $u$) and $1$ (in $10^41$). Since $u$ has a
prefix $1^2 0^4$, it means that there is a long right semi-abelian
return ending in $1^2 0^4$, i.e., we have right semi-abelian
return $\alpha'_2$ or $\alpha'_3$. A suffix $0^4 1^2$ of $u$
implies that there is a long right semi-abelian return $\alpha'_1$
or $\alpha'_2$. The only possibility is that an abelian class of
$1^2 0^2$ has abelian returns $0$, $1$ and $\alpha'_2$ with
$j_2\geq0$, and nothing else. The set of abelian returns  $0$, $1$
and $\alpha'_1\sim_{ab} \alpha'_3$ is impossible since in this
case the abelian class $1^2 0^2$ has other abelian returns than
$0$, $1$, $\alpha_2$. The factor $u$ has a suffix $0^4 1^2$, so
the factor $0^41$ here has right semi-abelian return $\alpha'_2$,
so $u$ is continued in the unique way: $u=1^2 0^2 (0^210^2)^{i_2}
0^2 1 (10^2 1)^{j_2} 0^2$. This factor has a suffix $1 0^2 1 0^2$,
but we proved above that in the case 2c) the factor $1 0^2 1$ is
always followed by $1$, so we get a contradiction. Hence $1^2 0^2$
has no returns of the form $\alpha_2$.

\smallskip

\noindent \textbf{Case 2d)} let $1 0^2 1$ have a return $\alpha_4$
with $i_4\geq 0$. Notice that if $1 0^2 1$ has only return $\alpha_4$, then $w=(0^2 1 (0^4 1)^{i_4})^{\omega}$, and $w$ does not contain the block $1^2$. We proved
that there is no long returns of the forms $\alpha_1$, $\alpha_2$
and $\alpha_3$, so the only possibility is that $10^21$ has two
returns $\alpha_4$ and $1$, and $1^20^2$ always has return $1$.
So, $10^21$ is followed by either $(0^4 1)^{i_4} 0^2 1 $ or $1$.
In the second case we can determine several next letters to be
$001$ (because and $11$ is always followed by $00$, and $1^20^2$
always has return $1$). So, $10^21$ is followed by either $(0^4
1)^{i_4} 0^2 1 $ or $10^21$. Both continuations have suffix
$10^21$, which is followed by either $(0^4 1)^{i_4} 0^2 1 $ or
$1$:

\begin{picture}(300,80)

\put(0,35){\makebox(50,10)[l]{$10^21$}}
\put(25,40){\line(1,1){15}} \put(25,40){\line(1,-1){15}}
{\put(45,50){\makebox(70,10)[l]{$(0^4 1)^{i_4} 0^2 1$}}
\put(95,55){\line(1,1){10}} \put(95,55){\line(1,-1){10}}}
{\put(45,20){\makebox(90,10)[l]{$1 0^2 1$}}
\put(70,25){\line(1,1){10}} \put(70,25){\line(1,-1){10}}}

\put(120,50){\makebox(20,10)[l]{$\dots$}}
\put(90,20){\makebox(20,10)[l]{$\dots$}}

\end{picture}

\noindent Thus $w\in \{ (0^4 1)^{i_4} 0^2 1, 1 0^21 \}^{\omega}$.
We are interested in the case when all four blocks are contained
in $w$, so we get $i_4>0$, otherwise $w$ does not contain the
block $0^4$.

Thus $w$ contains a factor $u=1 0^2 1 (0^4 1)^{i_4} 0^2 1$. Now
consider abelian returns to the abelian class of $0^4 1$. One can
find right semi-abelian returns $0$ (in a factor $0^2 1 0^3$ of
$u$) and $1$ (in $10^41$). Since $u$ has a prefix  $1 0^2 1 0^2$,
we have a long right semi-abelian return ending in $1 0^2 1 0^2$,
i.e., $\alpha'_1$ or $\alpha'_4$. A suffix  $0^2 1 0^2 1$ of $u$
implies that there is a long right semi-abelian return $\alpha'_3$
or $\alpha'_4$ with $j_4\geq0$. The only possibility is that an
abelian class of $0^4 1$ has abelian returns $0$, $1$ and
$\alpha'_4$ with $j_4\geq0$. The set of abelian returns $0$, $1$
and $\alpha'_1\sim_{ab}\alpha'_3$ is impossible since is this case
the abelian class of $0^2 1^2$ has other abelian returns than $0$,
$1$ and $\alpha_4$.
 Considering the suffix $0^2 1 0^2 1$ of $u$, we get that the
factor $0^210^2$ here has right semi-abelian return $\alpha'_4$,
so $u$ is continued in the unique way: $u'=1 0^2 1 (0^4 1)^{i_4}
0^2 (1^2 0^2)^{j_4} 10^2$.
 The factor $u'$ has a suffix $1 0^2 1 0^2$, so the factor $1 0^2 1$
 here has right semi-abelian return $\alpha_4$,
so it is continued in the unique way: $u''=1 0^2 1 (0^4 1)^{i_4}
0^2 (1^2 0^2)^{j_4} 1(0^4 1)^{i_4} 0^2 1$. Continuing this line of
reasoning, we obtain a periodic word $w$. Its factor $(0^4
1)^{i_4-1} 0^4$ has only one semi-abelian return. Hence $1 0^2 1$
has no long returns $\alpha_4$.

\medskip

So, we are done with the case 2)

\bigskip

\noindent \textbf{Case 3)}: $w\in (\{ 0, 0^2\}\{1,
1^2\})^{\omega}$.


Consider abelian returns for the word $0^2 1$. Factors of $w$ from
the abelian class of $0^2 1$ could be the following: $1 0^2$, $0^2
1$, $010$, and each of them necessarily appears in $w$.

\begin{itemize}

\item  $1 0^2$ has right semi-abelian return $1$

\item   $0^2 1$ has right semi-abelian returns of the form $\alpha_1 =
(101)^{i_1} 0$ and/or $\alpha_2 = (101)^{i_2} 1 0^2$ for some
$i_1, i_2\geq 0$.

\item   $010$ has right semi-abelian returns of the form $\alpha_3 =
(110)^{i_3} 0$ and/or $\alpha_4 = (110)^{i_4} 10$ for some $i_3,
i_4\geq 0$.

\end{itemize}

Symmetrically, we get possible abelian returns for $1^2 0$:

\begin{itemize}

\item   $0 1^2$ has right semi-abelian return $0$

\item   $1^2 0$ has right semi-abelian returns of the form $\alpha'_1 =
(010)^{j_1} 1$ and/or $\alpha'_2 = (010)^{j_2} 0 1^2$ for some
$j_1, j_2\geq 0$.

\item   $101$ has right semi-abelian returns of the form $\alpha'_3 =
(001)^{j_3} 1$ and/or $\alpha'_4 = (001)^{j_4} 01$ for some $j_3,
j_4\geq 0$.

\end{itemize}

\noindent These are summarized in the table below:

\medskip

\noindent \begin{tabular}{|c|l|l|} \hline
                 abelian class & word  & possible right semi-abelian returns  \\ \hline

 \multirow{3}{*}{$0^2 1$} & $1 0^2$ & $1$

\\ \cline{2-3}
& $0^2 1$ & $\alpha_1 = (101)^{i_1} 0$, $\alpha_2 = (101)^{i_2} 1
0^2$ for some $i_1, i_2\geq 0$
\\ \cline{2-3}
   &  $010$ & $\alpha_3 =
(110)^{i_3} 0$, $\alpha_4 = (110)^{i_4} 10$ for some $i_3, i_4\geq
0$
\\ \hline \hline

 \multirow{3}{*}{$1^2 0$} & $0 1^2$ & $0$

\\ \cline{2-3}
   & $1^2 0$ &  $\alpha'_1 =
(010)^{j_1} 1$, $\alpha'_2 = (010)^{j_2} 0 1^2$ for some $j_1,
j_2\geq 0$
\\ \cline{2-3}
& $101$ & $\alpha'_3 = (001)^{j_3} 1$, $\alpha'_4 = (001)^{j_4}
01$ for some $j_3, j_4\geq 0$

\\ \hline

\end{tabular}

\medskip

Notice that $\alpha_1 \sim_{ab} \alpha_3$ when $i_1=i_3$, and
$\alpha'_1 \sim_{ab} \alpha'_3$ when $j_1=j_3$. In this case
the lengths of blocks of $0$'s and $1$'s are the same, so we can
use symmetry in the proofs.

If factors from the abelian class of
$0^2 1$ have only letters as abelian returns, then $w=(0^2 1)^{\omega}$, and this word does not contain all four blocks. So, a factor from the abelian class of
$0^2 1$ should contain a long abelian return (of length longer
than $1$), so we consider the four cases corresponding to long
returns $\alpha_1$--$\alpha_4$.

\smallskip

\noindent \textbf{Case 3a)} let $0^2 1$ have a return $\alpha_1$
with $i_1>0$. Then $w$ contains a factor $u=0^2 1 (101)^{i_1} 0$.
Now consider abelian returns to the abelian class of $1^2 0$. One
can find right semi-abelian returns $1$ (in a factor $1101$) and
$0$ (in $0110$). Since $u$ has a prefix  $0 0 1^2$, it means that
there is a long right semi-abelian return ending in $0^2 1^2$, i.
e., $\alpha'_2$ or $\alpha'_3$. A suffix  $1 0 1 0 $ of $u$
implies that there is a long right semi-abelian return $\alpha'_3$
or $\alpha'_4$. So, the only possibility is that an abelian class
of $1^2 0$ has abelian returns $0$, $1$ and $\alpha'_3 \sim_{ab}
\alpha'_1$ with $j_1=j_3>0$. Considering the suffix $1 0 1 0 $ of
$u$, we get that the factor $101$ here has right semi-abelian
return $\alpha'_3$, so $u$ is continued in the unique way: $u'=0^2
1 (101)^{i_1} (001)^{j_3} 1$. One can find in $u'$ abelian returns
$0$ and $1$ to the abelian class of $0^2 1$, and we started with
the long return $\alpha_1\sim_{ab} \alpha_3$. The factor $u'$ has
a suffix $0^2 1^2$, so the factor $001$ here has right
semi-abelian return $\alpha_1$, and hence $u'$ is continued in the
unique way: $u''=0^2 1 (101)^{i_1} (001)^{j_3} (101)^{i_1} 0$.
Continuing this line of reasoning, we obtain a periodic word, in
which the abelian class of $1 (101)^{i_1}$ has one semi-abelian
return. Hence $0^2 1$ has no long returns $\alpha_1$, and
symmetrically $1^2 0$ has no long returns $\alpha'_1$.

\smallskip

\noindent \textbf{Case 3b) }let $0 1 0$ have a return $\alpha_3$
with $i_3>0$. Then $w$ contains a factor $u= 0 1 0 (110)^{i_3} 0$.
Now consider abelian returns to the abelian class of $1^2 0$. One
can find right semi-abelian returns $1$ (in a factor $1011$) and
$0$ (in $0110$). Due to the prefix  $0 1 0 1$ of $u$, there is a
long right semi-abelian return ending in $0 1 0 1$, i.e.,
$\alpha'_1$ or $\alpha'_4$. The suffix  $1 1 0 0 $ of $u$ implies
that there is a long right semi-abelian return $\alpha'_1$ or
$\alpha'_2$. We proved that there are no long returns of the form
$\alpha'_1$, so $1^2 0$ has right semi-abelian returns $0$, $1$,
$\alpha'_4$, $\alpha'_2$. None of them are abelian equivalent, a
contradiction. Hence $0^21$ has no returns of the form $\alpha_3$,
and symmetrically $1^20$ has no returns $\alpha'_3$.

\smallskip

\noindent \textbf{Case 3c) }let $0^2 1$ have a return $\alpha_2$.
The abelian class of $001$ always has abelian return $1$. If $0^2
1$ has only return $\alpha_2$, then $w= ( (1 0 1)^{i_2} 1 0^2 1
)^{\omega}$, and the factor $0^2$ has only one abelian return. So,
$0^2 1$ has also other abelian returns. Taking into account that
there are no long returns of the forms $\alpha_1$ and $\alpha_3$,
and $\alpha_2$ is never abelian equivalent to $\alpha_4$, we get
that there should be abelian return $0$. Hence, there is no
abelian return $\alpha_4$ and $010$ is always followed by $0$,
$0^21$ is followed by either $0$ or $\alpha_2$. So, $w$ contains a
factor $u= 0^2 1 (101)^{i_2} 1 0^2$, $i_2\geq 0$. Now consider
abelian returns to the abelian class of $1^2 0$. Since $u$ has a
prefix $0^2 1^2$, it means that there is a long right semi-abelian
return ending in $0^2 1^2$, i.e., we have right semi-abelian
return $\alpha'_2$ or $\alpha'_3$. A suffix $1^2 0^2$ of $u$
implies that there is a long right semi-abelian return $\alpha'_1$
or $\alpha'_2$. We proved that we never have long return
$\alpha'_1$, so we have right semi-abelian return $\alpha'_2$.
Symmetrically to what we proved above, we get that $101$ is always
followed by $1$, $110$ is followed by either $1$ or $\alpha'_2$.
So, the last occurrence of $110$ in $u$ is extended by
$\alpha'_2$, i.e. we get the unique extention of $u$: $u'= 0^2
1 (101)^{i_2} 1 0 (010)^{j_2} 0 1^2$. Considering the last
occurrence of the factor $001$ in $u'$, we get that it should have
right semi-abelian return $\alpha_2$, i.e. we get the unique
extention: $u''= 0^2 1 (101)^{i_2} 1 0 (010)^{j_2} 0 1
(101)^{i_2} 1 0^2$. Continuing this line of reasoning, we get a
periodic word, in which the factor $0 (010)^{j_2} 0$ has only one
semi-abelian return. Hence we have no returns of the form
$\alpha_2$ and $\alpha'_2$.

\smallskip

\noindent \textbf{Case 3d)} In the remaining case the word $0 1 0$
has returns $0$ and  $\alpha_4$ with $i_4\geq0$, and the word $1 0
1$ has returns $1$ and  $\alpha'_4$ with $j_4\geq0$. So, $w$
contains a factor $u= 010 (110)^{i_4} 10$. Considering the last
occurrence of $101$ in $u$, we see that it has return $\alpha'_4$,
so $u$ is extended in the following way: $010 (110)^{i_4} 1
(001)^{j_4} 01$. The last occurrence of $010$ in this word
necessarily has right semi-abelian return $\alpha_4$, so the word
is extended uniquely as follows: $010 (110)^{i_4} 1 (001)^{j_4} 0
(110)^{i_4} 10$. Continuing this line of reasoning, we get a
periodic word. In this word $i_4> 0$, otherwise we do not have
occurrences of the block $1^2$, and the abelian class of
$(110)^{i_4}1$ has only one semi-abelian return.

So, we are done with the case 3)

\bigskip

\noindent \textbf{Case 4)} $w\in (\{ 0^2, 0^4\}\{1^2,
1^4\})^{\omega}$

This case is considered in exactly the same way as the case 3) by
considering abelian returns to $0^4 1^2$ and $0^2 1^4$. The only
changes which should be done are doubling $0$'s and $1$'s
everywhere except returns of length $1$ (letters). \qed

\medskip

\begin{lemma}\label{l1-l2}
If $w\in \{ 0^{l_1}1, 0^{l_2}1 \}^{\omega}$, $0<l_1<l_2$ is a
recurrent word such that each of its factors has at most three abelian returns
and at least two semi-abelian returns, then $l_2=l_1+1$.
\end{lemma}

\noindent\emph{Proof.} Suppose that $l_2>l_1+1$. Consider abelian
returns to the word $0^{l_1+1}$: it has right abelian returns $0$
and $1(0^{l_1}1)^k 10^{l_1+1}$ for all $k\geq0$ such that $0^l_2 1
(0^{l_1}1)^k  0^{l_2}$ is a factor of $w$, thus there could be at
most two different values of $k$ (probably, including $0$).
Consider abelian returns to the word $10^{l_1}10$: it has right
abelian returns $0$ and $(0^{l_2-1}10)^j 0^{l_1-1}1$ for all
$j\geq0$ such that $1 0^{l_1} 1 (0^{l_2}1)^j 0^{l_1} 1$ is a
factor of $w$, thus there could be at most two different values of
$j$ (probably, including $0$). If we have only one value of $k$
and one value of $j$ simultaneously, then $w$ is periodic, $w=(
(0^{l_1}1)^{k_1} (0^{l_2}1)^{j_1})^{\omega}$. In this periodic
word if $k_1=0$, then the factor $0^l_2$ has one semi-abelian
return, if $k_1>0$, then the abelian class of $1(0^{l_1}1)^{k_1}$
has only one semi-abelian return. So, we have two cases:

\smallskip

\noindent Case I: $w\in (0^{l_2}1 \{ (0^{l_1}1)^{k_1},
(0^{l_1}1)^{k_2} \})^{\omega}$, $0<k_1<k_2$. In this case one can
find four abelian returns to $0^{l_2}10^{l_1-1}$: $0$,
$10^{l_1-1}$, $(10^{l_1})^{k_1-1} 10^{l_2-1}$, $(10^{l_1})^{k_2-1}
10^{l_2-1}$.

\smallskip

\noindent Case II: $w\in (0^{l_1}1 \{ (0^{l_2}1)^{j_1},
(0^{l_2}1)^{j_2} \})^{\omega}$, $0<j_1<j_2$. In this case one can
find four abelian returns to $10^{l_2}10^{l_1}10$: $0$,
$0^{l_2-1}1$, $(0^{l_2-1}10)^{j_1-1} 0^{l_1-1}1$,
$(0^{l_2-1}10)^{j_2-1} 0^{l_1-1}1$. \qed

\bigskip

The proofs of Lemma \ref{isolated} and Lemma \ref{l1-l2} immediately imply

\begin{corollary}\label{w} If each factor of an infinite binary recurrent word $w$
has at most three abelian returns and at least two semi-abelian
returns, then $w\in \{ 0^{l_1}1, 0^{l_1+1}1 \}^{\omega}$.
\end{corollary}

\begin{lemma}\label{2-balanced}
If each factor of a recurrent infinite binary word $w$ has at most
three abelian returns and at least two semi-abelian returns, then
$w$ is $2$-balanced.
\end{lemma}

\noindent\emph{Proof.} For a length $n$, consider abelian classes
of factors of length $n$ of $w$. Denote by $A$ the
abelian class of factors containing the smallest number of $1$-s:
$A=\{ u \in F_n(w): |u|_1=\min_{v\in F_n(w)} |v|_1 \}$. The next
class we denote by $B$: $B=\{ u \in F_n(w): |u|_1=\min_{v\in
F_n(w)} |v|_1+1 \}$, the next one by $C$. If $w$ has only two
abelian classes, then it is Sturmian, so we are interested in the
case when $w$ has at least three abelian classes. For a length
$n$, we associate to a word $w$ a word $\xi^{(n)}$ over the
alphabet of abelian classes of $w$ of length $n$ as follows: for
an abelian class $M$ of words of length $n$, $\xi^{(n)}_k=M$ iff
$w_k\dots w_{k+n-1}\in M$. In other words,
$(\xi^{(n)}_k)_{k\geq0}$ is the sequence of abelian classes of
consecutive factors of length $n$ in $w$.

It is easy to see that $\xi^{(n)}$ contains the following sequence
of classes: $CB^{j_1}A^{j_2}B$ for some $j_1, j_2\geq 1$, i.e.
for some $i$ we have $\xi^{(n)}_i\dots \xi^{(n)}_{i+j_1+j_2+1} =
CB^{j_1}A^{j_2}B$. Then we have
$$
\begin{aligned} &w_i =1, w_{i+n}=0,\\
&w_{k}=w_{k+n} \mbox{ for } k=i+1, \dots, i+j_1-1, \\
&w_{i+j_1} =1, w_{i+j_1+n}=0, \\
&w_{k}=w_{k+n} \mbox{ for } k=i+j_1+1, \dots, i+j_1+j_2, \\
&w_{i+j_1+j_2} =0, w_{i+j_1+j_2+n}=1. \end{aligned}
$$
I. e., $w_i
\dots w_{i+j_1+j_2} = 1 u 1 v 0$, $w_{i+n} \dots w_{i+j_1+j_2+n} =
0 u 0 v 1$.

By Corollary \ref{w} we have $w\in \{ 0^{l_1}1,
0^{l_1+1}1 \}^{\omega}$, so $|u|\geq 2l_1+1$; $u$ contains both
letters $0$ and $1$ and has a suffix $0^{l_1}$. It follows that
$j_2=1$. So, the class $B$ has the following $3$ abelian returns:
$0, 1, 01$. All the returns are of length at most $2$, so if after
an occurrence of $B$ we have $C$, then the next class is $B$
again, otherwise we will get a longer return. So there are no
other classes than these. In addition, we proved that if for
length $n$ there are three abelian classes, then in $\xi^{(n)}$
letters $A$ and $C$ are isolated. \qed


\bigskip

\noindent\emph{Proof of Proposition \ref {sufficiency}.} By
Corollary \ref {w} and Lemma \ref {2-balanced}, we have that $w$
is $2$-balanced and it is of the form $\{ 0^{l_1}1, 0^{l_1+1}1
\}^{\omega}$ for some integer $l_1$. Suppose that $w$ is not
balanced. Then there exists $n$ for which there exist three
classes of abelian equivalence in $F_n(w)$; as above, denote these
classes by $A$, $B$ and $C$. Arguing as in the proof of Lemma \ref
{2-balanced}, consider a sequence of classes $BCB^{j}AB$ which we
necessarily have in $\xi^{(n)}$ for some integer $j$, denote its
starting position by $i-1$. Corresponding factor in $w$ is
$$
\begin{aligned} &w_{i-1}=0, w_{i-1+n}=1, \\
&w_i =1, w_{i+n}=0,\\
&w_{k}=w_{k+n} \mbox{ for } k=i+1, \dots i+j-1, \\
&w_{i+j} =1, w_{i+j+n}=0, \\
&w_{i+j+1}=0, w_{i+j+1+n}=1. \end{aligned}
$$
I. e., $w_i \dots w_{i+j+1} = 1 u 10$, $w_{i+n} \dots w_{i+j+1+n}
= 0 u 01$. Remark that $u=w_{i+1}\dots w_{i+j}$ has prefix
$0^{l_1}10$.

Now consider abelian returns to an abelian class $B0=A1$ of length
$n+1$. The factor starting from the position $i+1$ is of the form
$B0$ so it belongs to this class, and has an abelian return $0$.
The word starting from the position $i+j$ is of the form $B0$ and
has an abelian return $1$. The word starting from the position
$i+l_1-1$ belongs to this class, and has an abelian return $10$.
So we have at least three returns $0$, $1$ and $10$. Now consider
the occurrence of class $B0=A1$ to the left from the position
$i+1$. One can see that the positions $i$ and $i-1$ are from the
class $B1=C0$, so the preceding occurrence of $B0=A1$ has an
abelian return of length greater than $2$, which is a fourth
return, though there should be at most three. So we cannot have
more than two classes of abelian equivalence in a binary word
having two or three abelian returns, i.e., such word should be
balanced. Proposition \ref {sufficiency} is proved.\qed

\begin{lemma}\label {1-balanced}Let $w \in \{0,1\}^\omega$ be a recurrent balanced word. Then $w$ is either Sturmian or periodic. In the latter case there exists a (possibly empty) bispecial factor $B$ of $\omega$ and a letter $a\in \{0,1\}$ such that $aBa$ is a factor of $w$ having exactly one first return in $w.$ Since $aBa$ is the unique element in its abelian class, it follows that if $w$ is periodic then $w$ contains a factor having only one semi-abelian return.
\end{lemma}

\begin{proof} Since $w$ is assumed balanced, $w$ contains at most one right special factor for each length $n.$ If $w$ is not Sturmian, then $w$ is ultimately periodic, and hence periodic since it is recurrent. From here on we shall assume that $w$ is periodic. Thus $w$ has only a finite number of right special factors.  As $w$ is recurrent,  the longest right special factor of $w$ is also a bispecial factor of $w.$  Let $\varepsilon=B_0,B_1,\ldots, B_n$ denote the bispecial factors of $w$ in order of increasing length.  Thus $B_n$ is also the longest right special factor of $w.$
Set $B=B_{n-1}.$ Then there exists a unique letter $a\in \{0,1\}$
such that $aB$ is a right special factor. In particular both $aBa$
and $bBa$ are factors of $w$ where $a\neq b\in \{0,1\}.$ We
claim that the only right special factor of $w$ which begins
in $Ba$ is $B_n.$ Clearly, $B_n$ is a right special factor
beginning in $Ba$ (since $Ba$ is left special and hence must
coincide with the prefix of $B_n$ of its same length). To see that
no other right special factor of $w$   begins in $Ba,$ let
$R$ denote the shortest right special factor of $w$ beginning
in $Ba.$ Then $R$ is also left special and hence bispecial. It
follows that $R=B_n.$ Since $B_n$ is also the longest right
special factor of $w$ the claim is established. Having
established the claim, it follows that
 $aBa$ has a unique first return in $w.$ If not,  there would exist a right special factor beginning in $aBa.$ From the previous claim it would follow that $aB_n$ is right special contradicting that $B_n$ is the longest right special factor.\end{proof}

\noindent We are now ready to prove the sufficiency condition:

\begin{corollary} \label {corollary_suf}
If each factor of a binary recurrent infinite word has two or
three abelian returns, then the word is Sturmian.
\end{corollary}

\noindent\emph{Proof.} Follows from Proposition \ref{sufficiency}
and Lemma \ref {1-balanced}. \qed




\begin{corollary} \label {aperiodic}
An aperiodic recurrent infinite word $w$ is Sturmian if and only
if each factor $u$ of $w$  has two or three abelian returns in $w.$
\end{corollary}

\noindent\emph{Proof.} Lemma \ref {periodic} implies
that an aperiodic word with $2$ or $3$ abelian returns must
necessarily be binary. \qed

\section{Proof of Theorem \ref {semi-abelian}}

In this section we prove the characterization of Sturmian words in
terms of semi-abelian returns.

\medskip

\noindent\emph{Proof of Theorem \ref {semi-abelian}.} We have that
for every factor in an infinite word the number of its
semi-abelian returns is not less than the number of abelian
returns. So, Proposition \ref{sufficiency} and Lemma \ref{1-balanced}
imply that if each factor of an infinite binary recurrent word has
two or three semi-abelian returns, then the word is Sturmian.

Now we prove that each factor of a Sturmian word has at most three
semi-abelian returns. Suppose that a factor $v$ of a Sturmian word
has more than three semi-abelian returns. By Proposition
\ref{necessity} this factor has at most three abelian returns, so
there are at least two semi-abelian returns which are abelian
equivalent. Due to Proposition \ref {aBb}, semi-abelian returns to
factors of Sturmian words are Christoffel words, i.e., letters or
words of the form $aBb$, so if we have more than three
semi-abelian returns to $v$, then there should be both returns
$0B1$ and $1B0$.

In the case $|v|\geq |0B1|$ the return $0B1$ is given by a factor
$0B1x1B0$ for some $x\in\{0,1\}^*$, where $0B1x$ is abelian equivalent to $v$. The return
$1B0$ is given by a factor $1B0y0B1$ for some $y\in\{0,1\}^*$, where $1B0y$ is abelian
equivalent to $v$. So, we have factors $1x1$ and $0y0$, where $x$
and $y$ are abelian equivalent, a contradiction with balance.

In the case $1<|v|<|0B1|$ we have a factor $z$ whose
(intersecting) prefix and suffix are $0B1$ and $1B0$, resp., and
another factor $z'$ of the same length whose prefix and suffix are
$1B0$ and $0B1$, resp. So $B$ should have $1$ and $0$ at the same
position.

If $|v|=1$, i.e., $v$ is a letter, it is easy to see that $v$ has
two semi-abelian returns.

Thus, two different semi-abelian returns of the same length
greater than $1$ are impossible. This concludes the proof. \qed

\bigskip

Similarly to Corollary \ref {aperiodic}, we get

\begin{corollary}
An aperiodic recurrent infinite word $w$ is Sturmian if and only
if each factor $u$ of $w$  has two or three semi-abelian returns in $w.$
\end{corollary}

\bibliographystyle{elsarticle-num}
\bibliography{<your-bib-database>}

\end{document}